\documentclass[11pt]{amsart}
\usepackage{latexsym,amssymb,amsmath,youngtab}
\usepackage{youngtab}
\usepackage{epsfig}
\usepackage{amsmath,amsthm,amssymb,amscd,MnSymbol}
\usepackage[mathscr]{eucal}
\usepackage{verbatim, hyperref}
\usepackage{color}
\newcommand*{\textlabel}[2]{%
  \edef\@currentlabel{#1}% Set target label
  \phantomsection% Correct hyper reference link
  #1\label{#2}% Print and store label
}

\newtheoremstyle{custom}% name
  {3pt}%      Space above
  {3pt}%      Space below
  {\slshape}%         Body font
  {}%         Indent amount (empty = no indent, \parindent = para indent)
  {\bfseries}% Thm head font
  {.}%        Punctuation after thm head
  { }%     Space after thm head: " " = normal interword space;
   {}%         Thm head spec (can be left empty, meaning `normal')
\theoremstyle{custom}
\newtheorem{theorem}{Theorem}[section]
\newtheorem{proposition}[theorem]{Proposition}

\newtheorem{proposition/definition}[theorem]{Proposition/Definition}
\newtheorem{lemma}[theorem]{Lemma}
\newtheorem{corollary}[theorem]{Corollary}
\newtheorem{conjecture}[theorem]{Conjecture}

\theoremstyle{definition}
\newtheorem{definition}[theorem]{Definition}

\newtheorem{example}[theorem]{Example}

\theoremstyle{remark}
\newtheorem{remark}[theorem]{Remark}

\newtheoremstyle{exercise}% name
  {3pt}%      Space above
  {6pt}%      Space below
  {}%         Body font
  {}%         Indent amount (empty = no indent, \parindent = para indent)
  {\bfseries}% Thm head font
  {:}%        Punctuation after thm head
  { }%     Space after thm head: " " = normal interword space;
   {}%         Thm head spec (can be left empty, meaning `normal')
\theoremstyle{exercise}
\newtheorem{exercise}[theorem]{Exercise}
\newtheoremstyle{exercises}% name
  {3pt}%      Space above
  {6pt}%      Space below
  {}%         Body font
  {}%         Indent amount (empty = no indent, \parindent = para indent)
  {\bfseries}% Thm head font
  {:}%        Punctuation after thm head
  {\newline}%     Space after thm head: " " = normal interword space;
   {}%         Thm head spec (can be left empty, meaning `normal')

\theoremstyle{exercise}
\newtheorem{exercises}[theorem]{Exercises}

\input epsf
\def\boxit#1{\vbox{\hrule height1pt\hbox{\vrule width1pt\kern3pt
  \vbox{\kern3pt#1\kern3pt}\kern3pt\vrule width1pt}\hrule height1pt}}

\def\BC{\mathbb C}

\def\11{\mathbf 1}

\def\o{\omega}

\def\s{\sigma}

\def\dim{{\rm dim}\;}

\def\s{\sigma}

\def\BC{\mathbb  C}

\def\tperm{\operatorname{perm}}

\def\be{\begin{equation}}
\def\ene{\end{equation}}

\DeclareMathOperator{\tlog}{log}

\def\dim{{\rm dim }}
\def\Ker{{\rm Ker }}
\DeclareMathOperator*{\perm}{perm}

\def\dim{{\rm dim }}
\begin{document}

\title{Equations for secant varieties of Chow varieties}
\author{Yonghui Guan}
 \begin{abstract}
The Chow variety of polynomials that decompose as a product of linear forms has been studied for more than 100 years.
Finding equations in the ideal of  secant varieties of Chow varieties would enable one to measure the complexity
the permanent  to prove Valiant's conjecture $\mathbf{VP\neq VNP}$. In this article, I use the method of prolongation to obtain equations for secant varieties of Chow varieties as $GL(V)$-modules.
 \end{abstract}
 \email{yonghuig@math.tamu.edu }
\maketitle
%\raisebox{1cm}
\section{Introduction}
\subsection{Motivation from algebraic geometry}
There has been substantial recent interest in the equations of certain algebraic varieties that encode
natural properties of polynomials (see e.g. \cite{MR2310544,LMsec,MR3169697,MR3081636,LWsecseg}). Such varieties are usually preserved by algebraic groups and it is a natural question to understand the module structures of the spaces of equations. One variety of interest
is the {\it Chow variety} of polynomials that decompose as a product of linear forms, which is defined by
$Ch_d(V)=\mathbb{P}\{z\in S^dV|z=w_1\cdots w_d\,{\rm\ for\ some\ } w_i \in V\}\subset\mathbb{P}S^dV,$
where $V$ be a finite-dimensional complex vector space and  $\mathbb{P}S^dV$ is the projective space of homogeneous polynomials of degree $d$ on the dual space $V^*$.

The ideal of the Chow variety of polynomials that decompose as a product of linear forms has been studied for over 100 years, dating back at least to  Gordon and Hadamard. Let $S^\delta(S^dV)$ denote the space of homogeneous polynomials of degree $\delta$ on $S^dV^*$. The {\it Foulkes-Howe} map $h_{\delta,d}:S^\delta(S^dV)\rightarrow S^d(S^\delta V)$ (see \S\ref{fhchow} for the definition) was defined by Hermite \cite{hermite} when $\dim\ V=2$, and Hermite proved the map is an isomorphism in his celebrated \lq\lq Hermite reciprocity\rq\rq. Hadamard \cite{MR1554881} defined the map in
general and observed that its kernel is $I_\delta(Ch_d(V^*))$, the degree $\delta$ component of the ideal
of the Chow variety. The conjecture that $h_{\delta,d}$ is always of maximal rank dating back to Hadamard \cite{MR1504330}  has become known as the
\lq\lq Foulkes-Howe conjecture\rq\rq \cite{MR0037276,MR983608}.
%J.M$\ddot{\rm u}$ler and M.Neunh$\ddot{\rm o}$fer)\cite{MR2172706}) proved $h_{4,4}$ is an isomorphism.
M\"uller and Neunh\"offer \cite{MR2172706} proved the conjecture is false by showing the map $h_{5,5}$ is not injective.  Brion \cite{MR1243152,MR1601139} proved the Foulkes-Howe conjecture is true asymptotically, giving an explicit, but very large
bound for $\delta$ in terms of $d$ and $\dim\ V$. We do not understand this map when $d>4$ (see \cite{MR1243152,MR1601139,MR0037276,MR1504330,MR983608,MR1651092}).

Brill and Gordon (see \cite{gkz,Gordon,MR2865915}) wrote down set-theoretic equations
for the Chow variety of degree $d+1$,  called \lq\lq Brill's equations\rq\rq. Brill's equations give a geometric derivation of set-theoretic equations for the Chow variety, I computed Brill's equations in terms of a $GL(V)$-module from a representation-theoretic perspective \cite{2015arXiv150802293G}, where $GL(V)$  denotes the {\it general linear group} of invertible linear maps from $V$ to $V$.

Let $W$ be a complex vector space and $X\subset \mathbb{P}{W^*}$ be an algebraic variety, define
$\sigma^0_r(X)={\bigcup_{p_1,\cdots,p_r\in X}\langle p_1,\cdots,p_r\rangle} \subset \mathbb{P}W^*,$
where $\langle p_1,\cdots,p_r\rangle$ denotes the  projective plane spanned by $p_1,\cdots,p_r$.
Define the $r$-th {\it secant variety} of $X$ to be
$\sigma_r(X)=\overline{\sigma^0_r(X)} \subset \mathbb{P}W^*,$ where the overline denotes closure in the Zariski topology.

Secant varieties of Chow varieties are invariant under the action of the group $GL(V)$, therefore their ideals are
 $GL(V)$-modules (see \S\ref{Gvariety}). Previously very little was known about
the ideals of secant varieties of Chow varieties, I obtained determinantal equations for these varieties in \cite{2015arXiv151000886G}.
In this article, I obtain equations for secant varieties of Chow varieties in terms of $GL(V)$-modules based on what we know about the ideal of
 Chow varieties.

\subsection{Motivation from complexity theory}

Leslie Valiant \cite{vali:79-3} defined in 1979 an algebraic analogue of the famous ${\mathbf{P}}$ versus ${\mathbf{NP}}$ problem (see Appendix  in \S\ref{appendix}).
%(see Appendix \ref{appendix} for the definitions of circuits,  classes ${\mathbf{VP}}$ and ${\mathbf{VNP}}$).
The class ${\mathbf{VP}}$ is an algebraic analog of the class ${\mathbf{P}}$, and the class ${\mathbf{VNP}}$ is an algebraic analog of the class ${\mathbf{VP}}$.  Valiant's Conjecture $\mathbf{VP\neq VNP}$ \cite{vali:79-3} may be rephrases as  \lq\lq there does not exist polynomial size circuit that computes the permanent\rq\rq, defined by ${\perm}_n=\sum_{\sigma\in\mathfrak{S_n}}x_{1\sigma(1)}x_{2\sigma(2)}\cdots x_{n\sigma(n)}\in S^n{\mathbb{C}^{n^2}}$, where $\mathfrak{S_n}$ is the symmetric group and $\mathbb{C}^{n^2}$ has a basis $\{x_{ij}\}_{1\leq i,j\leq n}$.
The readers can refer to Appendix  in \S\ref{appendix} to learn more about  circuits, complexity classes and Valiant's Conjecture.

Let $h_n$ and $g_n$ be two positive sequences, define $h_n=\omega(g_n)$ if $\lim_{n\rightarrow\infty}\frac{h_n}{g_n}=\infty$.

A geometric method to approach Valiant's conjecture implicitly proposed by Gupta, Kamath, Kayal and Saptharishicite \cite{DBLP:journals/eccc/GuptaKKS13} is to determine equations for certain secant varieties.
The following theorem appeared in \cite{MR3343444}, it is a geometric rephrasing of results in \cite{DBLP:journals/eccc/GuptaKKS13}.
\begin{theorem}\cite{DBLP:journals/eccc/GuptaKKS13,MR3343444}\label{chowvnp}  If for all but a finite number of $m$, for all $r,n$ with $rn<2^{\sqrt{m}\tlog(m) \o(1)}$,
$$[\ell^{n-m}\tperm_m]\not\in \s_r(Ch_{n}(\BC^{m^2+1})),$$

then Valiant's Conjecture $\mathbf{VP\neq VNP}$ \cite{vali:79-3} holds.
\end{theorem}

Theorem \ref{chowvnp}  motivated me to study the varieties $\s_r(Ch_d(V))$.
 Although the equations I obtain here cannot  separate $\mathbf{VP}$ from $\mathbf{VNP}$, the results come from a geometric perspective, and these are  the first low degree equations for secant varieties of Chow varieties, in addition to the non-classical equations
 obtained in \cite{2015arXiv151000886G}.

My results include
\begin{itemize}
\item Equations for $\sigma_2(Ch_3(C^{6*}))$ (Theorems \ref{secant2chow37} and \ref{secant2chow38}).
\item Equations for $\s_r(Ch_4(\mathbb{C}^{4r*}))$ (Theorem \ref{secantrchow4}).
\item Properties related to plethysm coefficients (Theorems  \ref{Lowweight} and \ref{plethysmeven}).
\item  Equations for $\s_r(Ch_d(\mathbb{C}^{dr*}))$ when $d$ is even (Theorem \ref{seccanteven})
\end{itemize}

\subsection{Results }
Let $X\subset W^{*}$ be an algebraic variety. Suppose we know the ideal of $X$, there is a systematic
method called prolongation (see \S\ref{prolongation1} for definition) to compute the ideal of $\sigma_r(X)$, but this method is difficult to implement. This method was studied by J. Sidman and S. Sullivant \cite{MR2541390},  and J.M. Landsberg and L. Manivel \cite{LMsec}.

For any partition $\lambda$, let $S_\lambda V$ be the irreducible $GL(V)$-module determined by the partition $\lambda$, for example $S_{(d)}V = S^dV$, while $S_{(1^d)}V = \Lambda^dV$ is
the $d$-th exterior power of $V$. The group $GL(V)$ has an induced action on $S^k(S^dV)$  (see \S\ref{Gvariety}), so $S^k(S^dV)$ a $GL(V)$-module, and $S^k(S^dV)$ can be decomposed  into a direct sum of irreducible $GL(V)$-modules, the multiplicity of $S_\lambda V$ in $S^k(S^dV)$ is the plethysm coeffcient  $p_\lambda(k, d)$.
To obtain equations for secant varieties, on one hand I compute prolongations directly via differential operators and representation theory. On the other hand, I rephrase prolongations and reduce computing prolongations to computing  polarization maps (see \S\ref{prolongation1}) via plethysm coefficients and Littlewood-Richardson coefficients (see \S \ref{lRr}). This gives a path towards obtaining equations for secant varieties of Chow varieties and other varieties.

Let $I_d(X)$ denote the degree $d$ component of the ideal of $X$.
For $d=3$,
\begin{theorem}\label{secant2chow37}
Let $\dim\ V\leq6$,
$I_7(\sigma_2(Ch_3(V^*)))=0.$
\end{theorem}

Also
\begin{theorem}\label{secant2chow38}
Let $\dim\ V\geq6$,
$S_{(5,5,5,5,3,1)}V\subset I_8(\sigma_2(Ch_3(V^*))).$
\end{theorem}

\newpage

For $d=4$,
\begin{theorem}\label{secantrchow4}
Consider  $\dim\ V\geq4r$,
$$S_{(6,6,4^{4r-2})}V\subset I_{4r+1}(\sigma_r(Ch_4(V^*)).$$
\end{theorem}

A partition is an even partition if all the components of the partition are even numbers. When $d$ is even, any even partition with length no more than $k$ has positive plethysm coefficients in $S^k(S^dV)$ \cite {MR2745569}.
\begin{theorem}\label{seccanteven}
 The isotypic component of $S_{({(2m+2)}^m,{(2m)}^{2mr-m})}V$ is in $I_{2mr+1}(\sigma_r(Ch_{2m}(V^*)))$.
 Moreover any module with even partition and smaller than $((2m+2)^{2m-1},2)$ {\rm(}with respect to the lexicographic order\ in\ {\rm \S\ref{lRr})} is in $ I_{2mr+1}(\sigma_r(Ch_{2m}(V^*)))$.
\end{theorem}
\subsection{Organization}
In $\S\ref{background}$, I review semi-standard tableaux, G-variety, the Little-Richardson rule, how to write down highest weight vectors of a $GL(V)$-module via raising operators, and the Foulkes-Howe map related to the ideal of the Chow variety $Ch_d(V^*)$. In $\S\ref{prolongation}$, I explain how to compute prolongations and multiprolongations of a $GL(V)$-module via differential operators and representation theory to obtain equations for $\s_r(Ch_d(V^*))$. In $\S\ref{degree3}$, I prove Theorems \ref{secant2chow37} and \ref{secant2chow38}. In $\S\ref{degree4}$, I prove Theorem \ref{secantrchow4}. In $\S\ref{degreeeven}$, I prove a theorem related to  plethysm coefficients of $S^{2m}(S^{2m+1}V)$ , and using this I prove Theorem \ref{seccanteven}. In $\S\ref{plethysm}$, I prove a property about plethysm coefficients. In \S\ref{appendix}, I include knowledge in computer science about ${\mathbf{P}}$ versus ${\mathbf{NP}}$ problem, circuits, complexity classes and Valiant's Conjecture
\subsection{Acknowledgement}
 I thank my advisor J.M. Landsberg for discussing all the details throughout this article. I thank C. Ikenmeyer and M. Michalek for discussing the plethysm coefficients.  Most of this work was done while the author was visiting the Simons Institute for the Theory of Computing, UC Berkeley for the {\it Algorithms and Complexity in Algebraic Geometry} program, I thank the Simons Institute for providing a good research environment.
\section{Preliminaries }\label{background}
\subsection{G-variety}\label{Gvariety}
I follow the notation in \cite[\S4.7]{MR2865915}.
\begin{definition}
Let $W$ be a complex vector space. A variety $X\subset\mathbb{P}W$ is called a {\it G-variety} if $W$ is a module for the group  $G$
and for all $g\in G$ and $x\in X$, $g\cdot x\in X.$
\end{definition}
G has an induced action on $S^dW^*$ such that for any $P\in S^dW^*$ and $w\in W$,
$g\cdot P(w)=P(g^{-1}\cdot w)$. $I_d(X)$ is a linear subspace of $S^dW^*$ that is invariant
under the action of $G$, therefore:
\begin{proposition}
If $X\subset\mathbb{P}W$ is a $G$-variety, then the ideal of $X$ is a $G$-submodule of $S^\bullet W^*:=\bigoplus_{d=0}^{\infty}S^dW^*$.
\end{proposition}
\begin{example}
The group $GL(V)$ has an induced action on $S^dV$ and $S^k(S^dV^*)$ similarly. $Ch_d(V)$ and its secant varieties
are invariant under the action of $GL(V)$, therefore they are $GL(V)$-varieties and their ideals are $GL(V)$-submodules of $S^\bullet (S^dV^*)=\bigoplus_{k=0}^{\infty}S^k(S^dV^*)$.
\end{example}

Let $X\subset\mathbb{P}W$ be a $G$-variety, and $M$ be an irreducible submodule of $S^\bullet W^*$, then either $M\subset I(X)$
or $M\cap I(X)=\emptyset$. Thus to test if $M$ gives equations for $X$, one only need to test one polynomial in $M$.
\subsection{Semi-standard tableaux }
I follow the notation in \cite{FH} and \cite{MR2865915}.
A partition $\lambda$ of an integer $d$ is $ \lambda=(\lambda_1,\cdots,\lambda_m)$
 with $\lambda_1\geq\cdots\geq\lambda_m>0$, $\lambda_j\in\mathbb{N}$ and $\sum_{i=1}^{m}\lambda_i=d$. We say $d$ is the {\it order} of $\lambda$ and $m$ is the {\it length} of $\lambda$. We often denote this by $\lambda\vdash d$. To
a partition $\lambda\vdash d$, we associate a {\it Young diagram}, which is a left aligned collection of boxes with $\lambda_i$ boxes in row $i$.

A {\it filling} of a Young diagram using the numbers $\{1,\cdots,l\}$ is an assignment of one number
to each box, with repetitions allowed. A filled Young diagram is called a {\it Young tableau}.
%A {\it standard filling} is one in which the entries are strictly increasing in the both the rows and
%columns, while
A {\it semi-standard} filling is one in which the entries are strictly increasing in the
columns and weakly increasing in the rows. {\it Semi-standard tableau} is
similarly defined.

Let $\lambda$ be a partition with order $kd$, a semi-standard tableau of {\it shape $\lambda$
and content} $k\times d$ is a semi-standard tableau associated to $\lambda$ and filled with $\{1,\cdots,k\}$
such that each $i\in\{1,\cdots,k\}$  appears $d$ times.
\subsection{The Little-Richardson rule and Pieri's rule}\label{lRr}
Let $\pi$ and $\mu$ be two partitions, the tensor product $S_\lambda V\otimes S_\mu V$ is a $GL(V)$-module.
The littlewood-Richardson coefficients $c_{\pi\mu}^{\nu}$ are defined to be the multiplicity of $S_\nu V$
in $S_\lambda V\otimes S_\mu V$, i.e. $S_\lambda V\otimes S_\mu V=\bigoplus_{\nu}c_{\pi\mu}^{\nu}S_\nu V$.

We order partitions {\it lexicographically}: $\lambda>\mu$ if the first nonvanishing $\lambda_i-\mu_i$ is positive.
Necessary conditions for  $c_{\pi\mu}^{\nu}$ to be positive are $|\nu|=|\pi|+|\mu|$ and $\nu$ is greater than
$\pi$ and $\mu$.

In particular $S_\lambda V\otimes S^d V=c_{\lambda,(d)}^{\nu}S_\nu V$.
\begin{theorem}{\rm(Pieri's\ rule)}
\begin{eqnarray*}
 c_{\lambda,(d)}^{\nu}=
\begin{cases}
1\ {\rm\ if\ \nu\ is\ obtained\ from\ \lambda\ by\ adding\ d\ boxes\ to} \\
\ \ \ {\rm  the\ rows\ of\ \lambda\ with\ no\ two\ in\ the\ same\ column};\\
0\ {\rm otherwise}.
\end{cases}
\end{eqnarray*}
\end{theorem}
\begin{example}
By Pieri's rule,
$$S^aV\otimes S^bV=\bigoplus_{0\leq t\leq s,s+t=a+b}S_{(s,t)}V.$$
$$S_{(d,d)}V\otimes S^{d^2-d}V=\bigoplus_{j=0}^{d}S_{(d^2-j,d,j)}V.$$
\end{example}
% For a given partition  and alphabet A, respectively let SYT(A) and
%SSYT(A) denote the sets of standard and semi-standard tableau lled by letters from A.
%When A is understood, we drop it from the notation.
\subsection{Highest weight vectors of modules in $S^k(S^dV)$ via raising operators}
I follow the notation in \cite{FH}.
The group $GL(V)$ has a natural action on $V^{\otimes d}$ such that
$g\cdot(v_1\otimes v_2\cdots\otimes v_d)=g\cdot v_1\otimes\cdots\otimes g\cdot v_d$.
Let dim $V=n$ and let $\{e_1,e_2,\cdots,e_n\}$ be a basis of $V$.
Let $B\subset GL(V)$ be the subgroup of upper-triangular matrices (a Borel subgroup).
For any partition $\lambda=(\lambda_1,\cdots,\lambda_n)$, let $S_\lambda V$ be the irreducible $GL(V)$-module determined by the partition $\lambda$.
For each  $S_\lambda V$, there is a unique line that is preserved by $B$, called a {\it highest weight line}.
Let $\mathfrak{gl}(V)$ be the Lie algebra of $GL(V)$,
there is an induced action of $\mathfrak{gl}(V)$ on $V^{\otimes d}$. For $X\in \mathfrak{gl}(V)$,
$$X.(v_1\otimes v_2\cdots\otimes v_d)=X.v_1\otimes v_2\cdots\otimes v_d+v_1\otimes X.v_2\otimes\cdots \otimes v_d+\cdots+v_1\otimes v_2\cdots \otimes v_{d-1}\otimes X.v_d.$$
Let $E^i_j\in\mathfrak{gl}(V)$ such that $E^i_j(e_j)=e_i$ and $E^i_j(e_k)=0$ when $k\neq j$. If $i<j$,  $E^i_j$ is called
 a raising operator; if  $i>j$,  $E^i_j$  is called a lowering operator.

A highest weight vector of a $GL(V)$-module is a weight vector that is killed by all raising operators.
Each realization of the module $S_\lambda V$ has a unique highest weight line. Let $W$ be a $GL(V$)-module, the multiplicity of $S_\lambda V$ in $W$ is equal to the dimension of the highest weight space with respect to the partition $\lambda$.

Define the weight space $W_{(a_1,\cdots,a_n)}$$\subset$ $ S^k(S^dV)$ to be the set of all the weight vectors whose weights are $(a_1,\cdots,a_n)$.
Note that  $S^dV$ has a natural basis $\{e_1^{\alpha_1}\cdots e_n^{\alpha_n}\}_{\alpha_1+\cdots+\alpha_n=d} $.
\begin{example}
$S_{(4,2)}V\subset S^3(S^2V)$ has multiplicity 1.
\begin{proof}
Let $v$ be a highest weight vector of $S_{(4,2)}V$. The weight space $W_{(4,2)}$ has a basis
$\{(e_1^2)^2(e_2^2),(e_1^2)(e_1e_2)^2\}$. Write
$v=a(e_1^2)^2(e_2^2)+b(e_1^2)(e_1e_2)^2$, then $E^1_2v=0$ implies
$(2a+2b)(e_1^2)^2(e_1e_2)=0$, therefore $a=-b$, so the multiplicity  of $S_{(4,2)}V$ in $S^3(S^2V)$
is 1.
\end{proof}
\end{example}
\begin{proposition}\label{chow2}
The highest weight vector $f$ of $S_{(2^k)}V\subset S^k(S^2V)$ is
determinant of the $k\times k$ matrix M with $M_{ij}=e_ie_j$ for $1\leq i,j\leq k$.
\end{proposition}
\begin{proof}
Since $S_{(2^k)}V\subset S^k(S^2V)$ is of multiplicity one, we only need to prove
$\det M$ is killed by all raising operators $E_{i+1}^{i}$ ($i=1,2,...,k-1$).
By symmetry, we only need to prove $\det M$ is killed by the raising operator $E_{2}^{1}$.
It is straightforward to verify $\det M$ is killed by the raising operator $E_{2}^{1}$.
\end{proof}
\begin{remark} By observation, $\sigma_{k}(Ch_2(V^*))\subset S^2V^*$ can be seen as the variety of symmetric matrices of rank at most $2k$, whose ideal is generated by $(2k+1)\times (2k+1)$ minors of the matrix. By Proposition \ref{chow2}, these $(2k+1)\times (2k+1)$ minors are corresponding to the module $S_{(2^{2k+1})}V\subset S^{2k+1}(S^2V)$, therefore  $S_{(2^{2k+1})}V$ is the generator of the ideal of $\sigma_{k}(Ch_2(V^*))$ for $k\geq1$.
\end{remark}
\begin{proposition}\label{hwvector732}
The highest weight vector $f$ of $S_{(7,3,2)}V\subset S^4(S^3V)$ is
\begin{eqnarray*}
f&=&(e_1^3)^2(e_1e_2^2)(e_2e_3^2)-2(e_1^3)^2(e_1e_2e_3)(e_2^2e_3)+(e_1^3)^2(e_1e_3^2)(e_2^3)-(e_1^3)(e_1^2e_2)^2(e_2e_3^2)\\
&+&2(e_1^3)(e_1^2e_2)(e_1^2e_3)(e_2^2e_3)-4(e_1^3)(e_1^2e_2)(e_1e_2^2)(e_1e_3^2)+0(e_1^3)(e_1^2e_3)(e_1e_2^2)(e_1e_2e_3)\\
&+&3(e_1^2e_2)^3(e_1e_3^2)+4(e_1e_2e_3)^2(e_1^2e_2)(e_1^3)-(e_1^3)(e_1^2e_3)^2(e_2^3)+3(e_1^2e_2)(e_1e_2^2)(e_1^2e_3)^2\\
&-&6(e_1^2e_2)^2(e_1^2e_3)(e_1e_2e_3).
\end{eqnarray*}
\end{proposition}
\begin{proof}
Let $f \in W_{(7,3,2)}\subset S^4(S^3V)$ be a weight vector. The weight space
$W_{(7,3,2)}\subset S^4(S^3V)$ has dimension 12. Write $f$ as a linear combination of the basis vectors
and apply $E_2^1$ and $E_3^2$ to $f$, we get two systems of linear equations. There is a unique solution up to scale.
\end{proof}
\begin{remark}
The module $S_{(7,3,2)}V$ cuts out $Ch_3(V^*)$ set-theoretically \cite{2015arXiv150802293G}.
\end{remark}
\begin{proposition}\label{hwvector5421}
The highest weight vector $f$ of $S_{(5,4,2,1)}V\subset S^4(S^3V)$ is
\begin{eqnarray}
f=e_2^2e_4h_1+e_1e_3e_4h_2+e_1e_2e_4h_3+e_1^2e_4h_4.
\end{eqnarray}
Here
\begin{eqnarray*}
h_4&=&(e_1^2e_2)(e_2^3)(e_1e_3^2)-(e_1e_2^2)^2(e_1e_3^2)-(e_1^2e_2)(e_1e_2e_3)(e_2^2e_3)\\
&+&(e_1^2e_3)(e_1e_2^2)(e_2^2e_3)-(e_1e_2^2)(e_1e_2e_3)^2-(e_1^2e_3)(e_1e_2e_3)(e_2^3),
\end{eqnarray*}
$h_3=-E^1_2h_4$, $h_1=\frac{1}{2}E^1_2E^1_2h_4$ is a highest weight vector of $S_{(5,2,2)}V\subset S^3(S^3V)$
and $h_2=E^2_3E^1_2h_4$ is a highest weight vector of $S_{(4,4,1)}V\subset S^3(S^3V)$.
\end{proposition}
\subsection{Foulkes-Howe map and the ideal of Chow variety}\label{fhchow}

I follow the notation in  \cite[\S8.6]{MR2865915}.
Define the {\it Foulkes-Howe map}
$FH_{\delta,d}:S^\delta (S^dV)\rightarrow S^d (S^\delta V)$ as follows:
First include $S^\delta (S^dV)\subset V^{\otimes\delta d}$. Next, regroup and
symmetrize the blocks to $(S^\delta V)^{\otimes d}$. Finally, thinking of $S^\delta V$
as a single vector space, symmetrize again to land in $S^\delta (S^dV).$

\begin{example}
 $FH_{2,2}(x^2\cdot y^2)=(xy)^2$, and $FH_{2,2}((xy)^2)=\frac{1}{2}[x^2\cdot y^2+(xy)^2]$.
\end{example}
$FH_{\delta,d}$ is a $GL(V)$-module map and Hadamard \cite{MR1554881} observed  and Howe rediscovered the following relationship between
Foulkes-Howe map and ideal of Chow variety.
\begin{proposition}{\rm (Hadamard \cite{MR1554881})}\label{idealchow}
$\Ker\ FH_{\delta,d}=I_\delta(Ch_d(V^*))$.
\end{proposition}

\begin{corollary}\label{difference}When  $\delta=d+1$,
$\Ker\ FH_{d+1,d}=I_{d+1}(Ch_d(V^*))$. Therefore as an abstract $GL(V)$-module,
$I_{d+1}(Ch_d(V^*))\supset S^{d+1} (S^dV)-S^{d} (S^{d+1}V)$.
\end{corollary}

\begin{proposition}\label{inj}{\rm (Hermite \cite{hermite}, Hadamard \cite{MR1504330}, J.M$\ddot{\rm u}$ler and M.Neunh$\ddot{\rm o}$fer)\cite{MR2172706})}
When $d=2,3,4$, $FH_{d,d}$ are injective and hence surjective.
\end{proposition}
\begin{proposition}\label{sur}({\rm T. McKay \cite{MR2394689}})
If $FH_{\delta,d}$ is surjective, then $FH_{\delta+1,d}$
is surjective.
\end{proposition}
So when $d=2,3,4$, $FH_{d+1,d}$ are surjective, and $I_{d+1}(Ch_d(V^*))= S^{d+1} (S^dV)-S^{d} (S^{d+1}V)$ as $GL(V)$- modules.
\section{Prolongations, multiprolongations and partial derivatives}\label{prolongation}
\subsection{Prolongations, multiprolongations and ideals of secant varieties}\label{prolongation1}
I study prolongations, multiprolongations and how they relate to
ideals of secant varieties. Let $W$ be a complex vector space with a basis $\{e_1,\cdots,e_n\}$.
%I follow the notation in \cite{MR2865915}.
%\begin{definition}
%Let $X\subset \mathbb{P}V^*$, define the $r$-th secant variety of X by
%$$\sigma_r(X)=\overline{\mathbb{P}\{v\in V|v=x_1+\cdots+x_r\ for\ x_i\in\hat{X}\ and\ i=1,\cdots,r\}}$$
%\end{definition}
\begin{definition}
For $A\subset S^dW$, define the $p$-th prolongation of $A$ to be:
$$A^{(p)}=(A\otimes S^pW)\cap S^{p+d}W.$$
It is equivalent to saying that
$$A^{(p)}=\{f\in S^{p+d}W|\frac{\partial^p f}{\partial e^\beta}\in A\ {\rm any}\ \beta\in \mathbb{N}^{n}\ {\rm with}\ |\beta|=p\}.$$
\end{definition}

For any $1\leq k\leq d$, there is an inclusion $F_{k,d-k}:S^dW\hookrightarrow S^kW\otimes S^{d-k}W$, called a {\it polarization map}.
Here are  properties of prolongation.
\begin{proposition}\label{propol}
For $A\subset S^dW$, $A^{(p)}$ is the inverse image of $A\otimes S^pW$ under the polarization map
$F_{d,p}:S^{d+p}W\rightarrow S^dW\otimes S^pW$.
\end{proposition}
\begin{proof}
For any $f\in S^{(p+d)}W$, $$F_{d,p}(f)=\sum_{|\alpha|=p}\frac{\partial^p f}{\partial e^\alpha}\otimes e^\alpha.$$ Hence
$$F_{d,p}(f)=\sum_{|\alpha|=p}\frac{\partial^p f}{\partial e^\alpha}\otimes e^\alpha \in A\otimes S^pW \Leftrightarrow \frac{\partial^p f}{\partial e^\alpha}\in A\ {\rm for\ any}\ |\alpha|=p \Leftrightarrow f \in A^{(p)}.$$
\end{proof}

\begin{theorem}\label{sidman}{\rm (J. Sidman, S. Sullivant \cite{MR2541390}})
Let $X\in\mathbb{P}W^*$ be an algebraic variety and let $d$ be the integer such that
$I_{d-1}(X)=0$ and $I_d(X)\neq0$. Then $I_{r(d-1)}(\sigma_r(X))=0$ and $I_{r(d-1)+1}(\sigma_r(X))=I_d(X)^{(r-1)(d-1)}.$
\end{theorem}
\begin{remark}
Theorem \ref{sidman} bounds the lowest degree of an element in the ideal of $\sigma_r(X)$ if we know  generators of the ideal of $X$.
\end{remark}
\begin{proposition}\label{proideal}
Let $X\subset\mathbb{P}W^*$ be an algebraic variety,
then $I_d(X)^{(p)}\subset I_{d+1}(X)^{(p-1)}$.
\end{proposition}
\begin{proof}
Let $f\in I_d(X)^{(p)}\subset S^{d+p}W$,
consider $\frac{\partial^{p-1} f}{\partial e^\alpha}$ with $|\alpha|=p-1$,
$$\frac{\partial^{p-1} f}{\partial e^\alpha}=\sum_{i=1}^{n}\frac{\partial^{p} f}{\partial (e^\alpha e_i)}e_i\in I_{d+1}(X).$$
\end{proof}
\begin{example}\label{I723}
Consider $Ch_3(V^*)$ with $\dim\ V\geq4$,  by Proposition \ref{inj} and Proposition \ref{sur},
$I_3(Ch_3(V^*))=0$ and
\begin{eqnarray}\label{I43}
I_4(Ch_3(V^*))=S^4(S^3V)-S^3(S^4V)=S_{(7,3,2)}V+S_{(6,2,2,2)}V+S_{(5,4,2,1)}V.
\end{eqnarray}
Therefore by Theorem \ref{sidman} $I_6(\sigma_2(X))=0$ and
$I_7(\sigma_2(X))=I_4(X)^{(3)}$.

\end{example}

The following proposition is about multiprolongations:
\begin{proposition}{\rm(Multiprolongation \cite{MR2865915} )}
Let $X\subset{P}W^*$ be an algebraic variety, a polynomial $P\in S^\delta W$ is in $I_\delta(\sigma_r(X))$ if and only if
for any  nonnegative decreasing sequence $(\delta_1,\delta_2,\cdots,\delta_r)$ with $\delta_1+\delta_2+\cdots+\delta_r=\delta$,
$$\bar{P}(v_1,\cdots,v_1,v_2,\cdots,v_2,\cdots,v_r,\cdots,v_r)=0$$
for all $v_i\in\hat{X}$, where the number of $v_i's$ appearing in the formula is $m_i$.
\end{proposition}

The following proposition rephrases multiprolongations.
\begin{proposition}\label{multipro}
Let $X\subset{P}W^*$ be an algebraic variety, for any positive integer $\delta$ and $r$,
and for any decreasing sequence $\vec{\delta}=(\delta_1,\delta_2,\cdots,\delta_r)$ with $\delta_1+\delta_2+\cdots+\delta_r=\delta$, consider the following polarization maps
$$F_{\delta_1,\delta_2,\cdots,\delta_r}:S^{\delta}W\rightarrow S^{\delta_1}W\otimes S^{\delta_2}W\otimes\cdots\otimes S^{\delta_r}W.$$
Let $A_{\vec{\delta},i}=S^{\delta_1}W\otimes\cdots\otimes S^{\delta_{i-1}}W\otimes I_{\delta_i}(X)\otimes S^{\delta_{i+1}}W\otimes\cdots\otimes S^{\delta_r}W
\subset S^{\delta_1}W\otimes S^{\delta_2}W\otimes\cdots\otimes S^{\delta_r}W$, then
$$I_\delta(\sigma_r(X))=\bigcap_{\delta_1+\delta_2+\cdots+\delta_r=\delta}F_{\delta_1,\delta_2,\cdots,\delta_r}^{-1}(A_{\vec{\delta},1}+\cdots+A_{\vec{\delta},r})$$
\end{proposition}

\begin{corollary}\label{proisec}
$I_d(X)^{((r-1)(d-1))}\subset I_{r(d-1)+1}(\sigma_r(X)).$
\end{corollary}

\begin{proof}
By Proposition \ref{multipro},
\begin{eqnarray*}
I_{r(d-1)+1}(\sigma_r(X))&=&\bigcap_{\delta_1+\delta_2+\cdots+\delta_r=r(d-1)+1,\ \delta_1\geq\delta_2\geq\cdots\geq\delta_r}F_{\delta_1,\delta_2,\cdots,\delta_r}^{-1}(A_{\vec{\delta},1}+\cdots+A_{\vec{\delta},r})\\
&\supset&\bigcap_{\delta_1+\delta_2+\cdots+\delta_r=r(d-1)+1,\ \delta_1\geq\delta_2\geq\cdots\geq\delta_r}F_{\delta_1,\delta_2,\cdots,\delta_r}^{-1}(A_{\vec{\delta},1}).\\
\end{eqnarray*}
By similar arguments as Proposition \ref{propol}, $F_{\delta_1,\delta_2,\cdots,\delta_r}^{-1}(A_{\vec{\delta},1})=I_{\delta_1}(X)^{(r(d-1)+1-\delta_1)}$.\\
Since $\delta_1\geq d$, by Proposition \ref{proideal}, $I_d(X)^{((r-1)(d-1))}\subset I_{\delta_1}(X)^{(r(d-1)+1-\delta_1)}$, therefore\\ $I_d(X)^{((r-1)(d-1))}\subset I_{r(d-1)+1}(\sigma_r(X)).$
\end{proof}

%\begin{remark}
%Proposition \ref{multipro} and Corollary \ref{proisec} implies Theorem \ref{sidman}.
%\end{remark}

\begin{proof}[A new proof of Theorem \ref{sidman}]
First, by Proposition \ref{multipro},
\begin{eqnarray*}
I_{r(d-1)}(\sigma_r(X))&=&\bigcap_{\delta_1+\delta_2+\cdots+\delta_r=r(d-1)}F_{\delta_1,\delta_2,\cdots,\delta_r}^{-1}(A_{\vec{\delta},1}+\cdots+A_{\vec{\delta},r}).
\end{eqnarray*}
In particular, when $\delta_1=\delta_2=\cdots=\delta_r=(d-1)$, $A_{\vec{\delta},i}=0$ for $i=1,\cdots,r$, so $F_{\delta_1,\delta_2,\cdots,\delta_r}^{-1}(A_{\vec{\delta},1}+\cdots+A_{\vec{\delta},r})=0$. Therefore $I_{r(d-1)}(\sigma_r(X))=0$.\\
Second, by Proposition \ref{multipro},
\begin{eqnarray*}
I_{r(d-1)+1}(\sigma_r(X))&=&\bigcap_{\delta_1+\delta_2+\cdots+\delta_r=r(d-1)+1}F_{\delta_1,\delta_2,\cdots,\delta_r}^{-1}(A_{\vec{\delta},1}+\cdots+A_{\vec{\delta},r}).
\end{eqnarray*}
In particular, when $\delta_1=d,\ \delta_2=\cdots=\delta_r=d-1$, $A_{\vec{\delta},i}=0$ for $i=2,\cdots,r$. so $$F_{\delta_1,\delta_2,\cdots,\delta_r}^{-1}(A_{\vec{\delta},1}+\cdots+A_{\vec{\delta},r})=F_{\delta_1,\delta_2,\cdots,\delta_r}^{-1}(A_{\vec{\delta},1})=I_d(X)^{((r-1)(d-1))}.$$ Therefore $I_{r(d-1)+1}(\sigma_r(X))\subset I_d(X)^{((r-1)(d-1))}$.\\
On the other hand, by Corollary \ref{proisec}, $I_d(X)^{((r-1)(d-1))}\subset I_{r(d-1)+1}(\sigma_r(X))$,
so equality holds.
\end{proof}
Theorem \ref{sidman}, small examples and intuition lead to the following conjecture:
\begin{conjecture}
Let $X\in{P}W^*$ be an algebraic variety, and $\delta=kr+l$ with $0\leq l<r$,
take $\vec{\delta}$ such that $\delta_1=\cdots=\delta_l=k+1$ and $\delta_{l+1}=\cdots=\delta_r=k$,
then $$I_\delta(\sigma_r(X))=F_{\delta_1,\delta_2,\cdots,\delta_r}^{-1}(A_{\vec{\delta},1}+\cdots+A_{\vec{\delta},r}).$$
\end{conjecture}
\begin{example}Consider $Ch_3(V^*)$, by Example \ref{I723},
$I_3(Ch_3(V^*))=0$ and $I_4(Ch_3(V^*))=S_{(7,3,2)}V+S_{(6,2,2,2)}V+S_{(5,4,2,1)}V$.
Consider the polarization maps
$$F_{\delta,8-\delta}:S^8(S^3V)\rightarrow S^{\delta}(S^3V)\otimes S^{8-\delta}(S^3V).$$
By Propositions \ref{multipro} and \ref{proideal},
\begin{eqnarray}
I_8(\sigma_2(Ch_3(V^*)))&=&\bigcap_{\delta=4}^8F_{\delta,8-\delta}^{-1}[S^{\delta}(S^3V)\otimes I_{8-\delta}(Ch_3(V^*))+I_{\delta}(Ch_3(V^*))\otimes S^{8-\delta}(S^3V)]\nonumber \\
&=&\bigcap_{\delta=5}^8I_{\delta}(Ch_3(V^*))^{(8-\delta)}\bigcap F_{4,4}^{-1}[I_4(Ch_3(V^*))\otimes S^{4}(S^3V)\nonumber\\
&&+ S^{4}(S^3V)\otimes I_4(Ch_3(V^*))]\nonumber \\
&=&I_{5}(Ch_3(V^*))^{(3)}\bigcap F_{4,4}^{-1}[I_4(Ch_3(V^*))\otimes S^{4}(S^3V)\nonumber\\
&&+ S^{4}(S^3V)\otimes I_4(Ch_3(V^*))].
\end{eqnarray}
\end{example}
%\section{Product maps and elongations}

\subsection{Partial derivatives and prolongations }

Let $V={\rm span} \{e_1,\cdots,e_n\}$,  $S^dV$ has a natural basis $\{e_1^{\alpha_1}\cdots e_n^{\alpha_n}:=e^\alpha\}_{\alpha_1+\cdots+\alpha_n=d}$. Assume $e_1>e_2>\cdots>e_n$. Define the {\it dominance partial order} on the natural basis of $S^dV$ such that
$$e^\alpha>e^{\beta}\Leftrightarrow \alpha_1+\cdots+\alpha_i\geq \beta_1+\cdots+\beta_i\ {\rm for\ each}\ i.$$
It is equivalent to saying
$$e^\alpha>e^{\beta}\Leftrightarrow\ {\rm one\ can\ get}\ e^\alpha\ {\rm from}\ e^{\beta}\ {\rm via\ raising\ operators}.$$

Let $f\in W_{(a_1,\cdots,a_n)}\subset S^k(S^dV)$, let $\alpha$ be the index of the last $d$ elements in $(a_1,\cdots,a_n)$,
then $\frac{\partial}{\partial e^\alpha}$ is the lowest possible partial derivative of $f$  with respect to the  dominance partial order.

\begin{example}
Let $f\in W_{(5,4,4,2)}\subset S^5(S^3V)$, then $\alpha=(0,0,1,2)$  and the lowest possible partial derivative of $f$ is $\frac{\partial f}{\partial e_3e_4^2}$.
\end{example}
\begin{definition}
Let $e^\alpha=e_1^{\alpha_1}\cdots e_j^{\alpha_j}e_{j+1}^{\alpha_{j+1}}\cdots e_n^{\alpha_n}$, for $j=1,\cdots,n-1$, define the normalized  lowering operators
 $$\tilde{ E}_{j}^{j+1}e^\alpha=e_1^{\alpha_1}\cdots e_j^{\alpha_j-1}e_{j+1}^{\alpha_{j+1}+1}\cdots e_n^{\alpha_n}.$$

\end{definition}

The following proposition gives the relationship between raising operators and partial derivatives of polynomials in $S^k(S^dV)$.

\begin{proposition}
Let $f\in S^k(S^dV)$ and $e^\alpha$ be a basis vector of $S^{d}V$,
then $$[\frac{\partial}{\partial e^\alpha},E_{j+1}^j]f=(1+\alpha_{j+1})\frac{\partial f}{\partial(\tilde{ E}_{j}^{j+1}e^\alpha)}.$$
Where $\tilde{ E}_{j}^{j+1}$($j=1,\cdots,n-1$) are the normalized  lowering operators.
\end{proposition}
\begin{proof}
Since all the operators here are linear, we only to prove the case when $f$ is a monomial.
Let $e^\alpha=e_1^{\alpha_1}\cdots e_j^{\alpha_j}e_{j+1}^{\alpha_{j+1}}\cdots e_n^{\alpha_n}$, so $\tilde{E}_{j}^{j+1}e^\alpha=e_1^{\alpha_1}\cdots e_j^{\alpha_j-1}e_{j+1}^{\alpha_{j+1}+1}\cdots e_n^{\alpha_n}=e^\beta$.
Write $f=g(e^\alpha)^m(e^\beta)^n$, where $g$ is not divisible by $e^\alpha$ or $e^\beta$.
Then
\begin{eqnarray*}
E_{j}^{j+1}f&=&(E_{j}^{j+1}g)(e^\alpha)^m(e^\beta)^n+gE_{j}^{j+1}((e^\alpha)^m)(e^\beta)^n+g(e^\alpha)^{m}E_{j}^{j+1}((e^\beta)^{n})\\
&=&(E_{j}^{j+1}g)(e^\alpha)^m(e^\beta)^n+mg(e^\alpha)^{m-1}E_{j}^{j+1}(e^\alpha)(e^\beta)^n+n(1+\alpha_{j+1})g(e^\alpha)^{m+1}(e^\beta)^{n-1}.
\end{eqnarray*}
So
\begin{eqnarray}\label{left}
\frac{\partial (E_{j}^{j+1}f)}{\partial e^\alpha}&=&m(E_{j}^{j+1}g)(e^\alpha)^{m-1}(e^\beta)^n+m(m-1)g(e^\alpha)^{m-2}E_{j}^{j+1}(e^\alpha)(e^\beta)^n+\nonumber\\
& &n(m+1)(1+\alpha_{j+1})g(e^\alpha)^{m}(e^\beta)^{n-1}.
\end{eqnarray}
On the other hand
$$\frac{\partial f}{\partial e^\alpha}=mg(e^\alpha)^{m-1}(e^\beta)^n.$$
\begin{eqnarray}\label{right}
E_{j}^{j+1}(\frac{\partial f}{\partial e^\alpha})&=&m (E_{j}^{j+1}g)(e^\alpha)^{m-1}(e^\beta)^n+m(m-1)g(e^\alpha)^{m-2}E_{j}^{j+1}(e^\alpha)(e^\beta)^n+\nonumber\\
& &nm(1+\alpha_{j+1})g(e^\alpha)^{m}(e^\beta)^{n-1}.
\end{eqnarray}
Combining \eqref{left} and \eqref{right}, we conclude:
\begin{eqnarray*}
\frac{\partial (E_{j}^{j+1}f)}{\partial e^\alpha}-E_{j+1}^j(\frac{\partial f}{\partial e^\alpha})=n(1+\alpha_{j+1})g(e^\alpha)^{m}(e^\beta)^{n-1}=(1+\alpha_{j+1})\frac{\partial f}{\partial(\tilde{E}_{j}^{j+1}e^\alpha)}.
\end{eqnarray*}
\end{proof}

In particular if $f\in S^k(S^dV)$ is a highest weight vector of some $GL(V)$-module, then
\begin{eqnarray}\label{raipar}
E_{j+1}^j(\frac{\partial f}{\partial e^\alpha})=-(1+\alpha_{j+1})\frac{\partial f}{\partial(\tilde{ E}_{j}^{j+1}e^\alpha)}.
\end{eqnarray}

Therefore
\begin{lemma}
If $f\in S^{k+1}(S^dV)$ is a highest weight vector for some $GL(V)$ module $S_{(a_1,\cdots,a_n)}V=S_aV$,
then the lowest possible partial derivative  $\frac{\partial f}{\partial e^\alpha}$  is killed by all the raising operators,
i.e. either  $\frac{\partial f}{\partial e^\alpha}$ is 0 or a highest weight vector of $S_{a-\alpha}V \subset S^k(S^dV)$.
\end{lemma}

By induction on dominance partial order, I conclude
\begin{proposition}\label{testnot}
If $f\in S^{k+1}(S^dV)$ is a highest weight vector for some module $S_{(a_1,\cdots,a_n)}V=S_aV$, then there exists
 a basis vector $e^\beta$ of $S^dV$ such that $\frac{\partial f}{\partial e^\beta}$ is a highest vector
of $S_{a-\beta}V \subset S^k(S^dV) $.
\end{proposition}
By Proposition \ref{testnot},
\begin{corollary}Let $f\in S^{k+1}(S^dV)$ be a highest weight vector for some module $S_{(a_1,\cdots,a_n)}V=S_aV$, if we can find all the $e^\beta$ such that $\frac{\partial f}{\partial e^\beta}$ is a highest vector
of $S_{a-\beta}V \subset S^k(S^dV) $, the sum of all these modules is the smallest possible module such that  $S_aV$ lies in its first prolongation.
\end{corollary}
For simplicity, write $\frac{\partial f}{\partial e^\beta}=f_{e^\beta}$ from now on.
\begin{example}
Let $f$ be the highest weight vector of $S_{(7,3,2)}V\subset S^4(S^3V)$ in Example \ref{hwvector732},
then $f_{e_2e_3^2}=(e_1^3)^2(e_1e_2^2)-e_1^3(e_1^2e_2)^2$, which is a highest weight vector of $S_{(7,2)}V\subset S^3(S^3V)$.
\end{example}

The following proposition, tells us  which prolongation a given module lies in.
\begin{proposition}\label{testyes}
If $S_aV\subset S^{k+1}(S^dV)$ with multiplicity $m_a>0$,
let
\begin{eqnarray*}
M_a&=&\{b|S_aV\subset S_bV\otimes S^dV\  {\rm as\ abstract\ modules\ by\ Pieri's\ rule}\\
&&{\rm and}\ S_bV\subset S^k(S^dV)\ {\rm with\ multiplicity}\ m_b>0\}.
\end{eqnarray*}
then $$(S_aV)^{\oplus m_a}\subset(\bigoplus_{b\in M_a}(S_bV)^{\oplus m_b})^{(1)}.$$

In particular, $$m_a\leq\sum_{b\in M_a}m_b.$$
\end{proposition}
\begin{proof}
Consider the polarization map $$P_{k,1}:S^{k+1}(S^dV) \rightarrow S^{k}(S^dV)\otimes S^dV.$$
By Schur's lemma $$P_{k,1}((S_aV)^{\oplus m_a})\subset (\bigoplus_{b\in M_a}(S_bV)^{\oplus m_b})\otimes S^dV.$$
By Proposition \ref{propol} $$(S_aV)^{\oplus m_a}\subset(\bigoplus_{b\in M_a}(S_bV)^{\oplus m_b})^{(1)}.$$
Since $P_{k,1}$ is injective,  $$m_a\leq\sum_{b\in M_a}m_b.$$
\end{proof}
\begin{proposition}\label{S5442}
 The module $S_{(5,4,4,2)}V\subset S^5(S^3V)$ is contained in $(S_{(5,4,2,1)}V\oplus S_{(4,4,4)}V)^{(1)}$.
Let $f\in S_{(5,4,4,2)}V\subset S^5(S^3V)$ be a highest weight vector,
then  $f_{e_1e_4^2} $ is a highest weight vector of $S_{(4,4,4)}V\subset S^4(S^3V)$ and
$f_{e_3^2e_4} $ is a highest weight vector of $S_{(5,4,2,1)}V\subset S^4(S^3V)$.
Therefore $S_{(5,4,4,2)}V$ is not contained in the first prolongation of $S_{(4,4,4)}V$ or $S_{(5,4,2,1)}V$.
\end{proposition}
\begin{proof}
Since
{\small
\begin{eqnarray*}
 S^4(S^3V)&=&S_{(12)}V +S_{(10,2)}V +S_{(9,3)}V +S_{(8,4)}V+\\
& &S_{( 8,2,2)}V+S_{(7,4,1)}V +S_{( 7,3,2)}V +S_{(6,6)}V +\\
& &S_{( 6,4,2)}V +S_{(6,2,2,2)}V +S_{(5,4,2,1)}V +S_{(4,4,4)}V.
\end{eqnarray*}}
By Proposition \ref{testyes}, $S_{(5,4,4,2)}\subset(S_{(5,4,2,1)}V\oplus S_{(4,4,4)}V)^{(1)}$.
By induction on the dominance partial order, $f_{e_1e_4^2} $ and $f_{e_3^2e_4} $ are killed by all
raising operators.  Let $h_1$ be a highest weight vector  of $S_{(4,4,4)}V\subset S^4(S^3V)$ and
$h_2$ be a highest weight vector  of $S_{(5,4,2,1)}V\subset S^4(S^3V)$.
Set $f_{e_1e_4^2}=c_1h_1$ and $f_{e_3^2e_4}=c_2h_2$, where $c_1$ and $c_2$ are constants, then $c_1, c_2$ can not be both 0 by Proposition \ref{testnot}.

Since $f_{e_3^3}\in S_{(5,4,2,1)}V\oplus S_{(4,4,4)}V$ with weight (5,4,1,2), $f_{e_3^3}=c_3E^4_3f_{e_3^2e_4}$, where $c_3$ is a constant. By \eqref{raipar}, $E^3_4f_{e_3^3}=-f_{e_3^2e_4}$, so
$c_3E^3_4E^4_3f_{e_3^2e_4}=-f_{e_3^2e_4}$, which implies $c_3(E_3^3-E_4^4)f_{e_3^2e_4}=-f_{e_3^2e_4}$,
so $c_3=-1$.
Since  $(f_{e_1e_4^2})_{e_3^3}=(f_{e_3^3})_{e_1e_4^2}$ ,
\begin{eqnarray*}
c_1(h_1)_{e_3^3}&=&(-E^4_3f_{e_3^2e_4})_{e_1e_4^2}\\
&=&-c_2(E^4_3h_2)_{e_1e_4^2}\\
&=&-c_2(E^4_3(h_2)_{e_1e_4^2}-(h_2)_{e_1e_3e_4})\\
&=&c_2(h_2)_{e_1e_3e_4}
\end{eqnarray*}
By Proposition \ref{hwvector5421}, $(h_1)_{e_3^3}$ and ($h_2)_{e_1e_3e_4}$ are both  highest weight vectors  of $S_{(4,4,1)}V\subset S^4(S^3V)$,
by rescaling, we may assume they are equal, so $c_1=c_2$, so $c_1$ and $c_2$ are both nonzero, therefore $f_{e_1e_4^2} $ is a highest weight vector of $S_{(4,4,4)}V\subset S^4(S^3V)$ and
$f_{e_3^2e_4} $ is a highest weight vector of $S_{(5,4,2,1)}V\subset S^4(S^3V)$.
\end{proof}
\section{The case when the degree is 3}\label{degree3}
%For $d=3$,
%\begin{theorem}\label{secant2chow37}
%Let $\dim\ V\leq6$,
%$I_7(\sigma_2(Ch_3(V^*)))=0.$
%\end{theorem}

%Also
%\begin{theorem}\label{secant2chow38}
%Let $\dim\ V\geq6$,
%$S_{(5,5,5,5,3,1)}V\subset I_8(\sigma_2(Ch_3(V^*))).$
%\end{theorem}

%\newpage

%For $d=4$,
%\%begin{theorem}\label{secantrchow4}
%Consider  $\dim\ V\geq4r$,
%$$S_{(6,6,4^{4r-2})}V\subset I_{4r+1}(\sigma_r(Ch_4(V^*)).$$
%\end{theorem}
Consider $\sigma_2(Ch_3(V^*))$, without loss of generality we assume $\dim\ V=6$.
\begin{proposition}\label{Pro3}
\begin{eqnarray*}
I_4(Ch_3(V^*))^{(1)}&=&S_{(7, 2, 2, 2, 2)}V\oplus S_{(6, 4, 2, 2, 1)}V\oplus S_{(5, 5, 3, 1, 1)}V,\\
I_4(Ch_3(V^*))^{(2)}&=&S_{(8, 2, 2, 2, 2, 2)}V\oplus S_{(7, 4, 2, 2, 2, 1)}V\oplus S_{(6, 5, 3, 2, 2, 1)}V\oplus S_{(5, 5, 5, 1, 1, 1)}V,\\
I_4(Ch_3(V^*))^{(3)}&=&0.
\end{eqnarray*}
\end{proposition}
\begin{proof}
First we claim
\begin{eqnarray}
I_4(Ch_3(V^*))^{(1)}=S_{(7, 2, 2, 2, 2)}V\oplus S_{(6, 4, 2, 2, 1)}V\oplus S_{(5, 5, 3, 1, 1)}V.
\end{eqnarray}
By \eqref{I43}, $$I_4(Ch_3(V^*))=S_{(7,3,2)}V+S_{(6,2,2,2)}V+S_{(5,4,2,1)}V.$$
By computer softwares (e.g. Lie), {\small
\begin{eqnarray*}
S^4(S^3V)&=&S_{(12)}V +S_{(10,2)}V +S_{(9,3)}V +S_{(8,4)}V+\\
& &S_{( 8,2,2)}V+S_{(7,4,1)}V +S_{( 7,3,2)}V +S_{(6,6)}V +\\
& &S_{( 6,4,2)}V +S_{(6,2,2,2)}V +S_{(5,4,2,1)}V +S_{(4,4,4)}V.
\end{eqnarray*}}
and
{\small
\begin{eqnarray*}
S^5(S^3V)&=&S_{(15)}V+S_{(13,2)}V+S_{(12,3)}V+S_{(11,4)}V+S_{(11,2,2)}V+S_{(10,5)}V+S_{(10,4,1)}V+S_{(10,3,2)}V+\\
& &S_{(9,6)}V+2S_{(9,4,2)}V+S_{(9,2,2,2)}V+S_{(8,6,1)}V+S_{(8,5,2)}V+S_{(8,4,3)}V+S_{(8,4,2,1)}V+\\
& &S_{(8,3,3,2)}V+S_{(7,6,2)}V+S_{(7,5,2,1)}V+S_{(7,4,4)}V+S_{(7,4,3,1)}V+S_{(7,4,2,2)}V+\\
& &S_{(7,2,2,2,2)}V+S_{(6,6,3)}V+S_{(6,5,2,2)}V+S_{(6,4,4,1)}V+S_{(6,4,2,2,1)}V+S_{(5,5,3,1,1)}V+S_{(5,4,4,2)}V.
\end{eqnarray*}}
Since $I_4(Ch_3(V^*))$ contains all the modules with length 4 in $S^4(S^3V)$, by Proposition \ref{testyes} any module with length 5 in $S^5(S^3V)$ is in
$I_4(Ch_3(V^*)^{(1)}$.

On the other hand,  the other modules with length no more than 4 in $S^5(S^3V)$ are not in $I_4(Ch_3(V^*)^{(1)}$:
By Proposition \ref{testnot}, for any  module  with length no more than 4 in $S^5(S^3V)$,  one can find a partial derivative of a highest weight vector of this module such that it is a highest weight vector of a module in $S^4(S^3V)$ but not in $I_4(Ch_3(V^*)$. For most modules, we can check directly, but for some modules, we need to verify carefully. For example, By Proposition \ref{S5442}, $S_{(5,4,4,2)}\subset(S_{(5,4,2,1)}\oplus S_{(4,4,4)}V)^{(1)}$, but $f_{e_1e_4^2}$ is a highest weight vector  of $S_{(4,4,4)}V\subsetneq I_4(Ch_3(V^*))$, so $S_{(5,4,4,2)}$ is not
not in $I_4(Ch_3(V^*)^{(1)}$. I conclude
$$I_4(Ch_3(V^*))^{(1)}=S_{(7, 2, 2, 2, 2)}V\oplus S_{(6, 4, 2, 2, 1)}V\oplus S_{(5, 5, 3, 1, 1)}V.$$
Similarly, by studying the modules in $S^6(S^3V)$ and $S^7(S^3V)$, we conclude
\begin{eqnarray*}
I_4(Ch_3(V^*))^{(2)}&=&S_{(8, 2, 2, 2, 2, 2)}V\oplus S_{(7, 4, 2, 2, 2, 1)}V\oplus S_{(6, 5, 3, 2, 2, 1)}V\oplus S_{(5, 5, 5, 1, 1, 1)}V,\\
I_4(Ch_3(V^*)^{(3)}&=&0.
\end{eqnarray*}
\end{proof}
Therefore by Proposition \ref{Pro3} and Theorem \ref{sidman},
\begin{theorem}{\rm (restatement of Theorem \ref{secant2chow37})}
$I_7(\sigma_2(Ch_3(V^*)))=I_4(Ch_3(V^*))^{(3)}=0.$
\end{theorem}

Also
\begin{theorem}{\rm (restatement of Theorem \ref{secant2chow38})}
$I_8(\sigma_2(Ch_3(V^*)))\supset S_{(5,5,5,5,3,1)}V.$
\end{theorem}
\begin{proof}
 By Example \ref{multipro}, $I_8(\sigma_2(Ch_3(V^*)))=I_{5}(Ch_3(V^*))^{(3)}\bigcap F_{4,4}^{-1}[I_4(Ch_3(V^*))\otimes S^{4}(S^3V)+ S^{4}(S^3V)\otimes I_4(Ch_3(V^*))].$
Since all the modules with 5 columns in $S^5(S^3V)$ are  contained in $I_{5}(Ch_3(V^*))$, by Proposition \ref{propol} and Schur's lemma,
\begin{equation}
S_{(5,5,5,5,3,1)}V\subset I_{5}(Ch_3(V^*)^{(3)}.
\end{equation}
Consider the map
$$F_{4,4}:S^8(S^3V)\rightarrow S^4(S^3V)\otimes S^{4}(S^3V).$$
Let ${I_4(Ch_3(V^*))}^{c}$ denote the complement to ${I_4(Ch_3(V^*))}$ in $S^{4}(S^3V)$.
Since
\begin{eqnarray*}
{I_4(Ch_3(V^*))}^{c}&=& S_{(12)}V + S_{(10,2)}V + S_{(9,3)}V + S_{(8,4)}V+\\
& & S_{( 8,2,2)}V+ S_{(7,4,1)}V  + S_{(6,6)}V + S_{( 6,4,2)}V+ S_{(4,4,4)}V,
\end{eqnarray*}
and $S_{(5,5,5,5,3,1)}V \nsubseteq S_{(4,4,4)}V\otimes S_{(4,4,4)}V$, by the Littlewood-Richardson rule,
$$S_{(5,5,5,5,3,1)}V \nsubseteq {I_4(Ch_3(V^*))}^{c}\otimes{I_4(Ch_3(V^*))}^{c}.$$
Therefore by Schur's lemma $$S_{(5,5,5,5,3,1)}V \subset  F_{4,4}^{-1}(I_4(Ch_3(V^*))\otimes S^{4}(S^3V)+ S^{4}(S^3V)\otimes I_4(Ch_3(V^*))).$$
The result follows.
\end{proof}

\begin{remark}
Since $\sigma_2(Ch_3(\mathbb{C}^{5*}))$ is a proper subset of $ \mathbb{P}S^3(\mathbb{C}^{5*})$, by inheritance (see \cite{MR2865915}), the ideal of $\sigma_2(Ch_3(V^*))$ should contain modules with length 5. So $S_{(5,5,5,5,3,1)}V$ is not enough to cut out $\sigma_2(Ch_3(V^*))$ set-theoretically. One can get length 5 modules with high degree in the ideal of  $\sigma_2(Ch_3(V^*))$ by Koszul Young flattenings \cite{2015arXiv151000886G}, but I still do not know whether they are enough to define $\sigma_2(Ch_3(V^*))$ set-theoretically. We know that $\dim\ S_{(5,5,5,5,3,1)}V=1134$  and ${\rm codim}\ \sigma_2(Ch_3(V^*))=24$, therefore $\sigma_2(Ch_3(V^*))$ is very far from being a complete intersection. Obviously $\mathbb{P}S^3(\mathbb{C}^{5*})$ with dimension 34 is in the zero set of $S_{(5,5,5,5,3,1)}V$, while the dimension of $\sigma_2(Ch_3(V^*))$ is 31,
the next question is: what is the difference between the dimension of $\sigma_2(Ch_3(V^*))$ and the zero set of $S_{(5,5,5,5,3,1)}V$?
\end{remark}
\section{The case when the degree is 4}\label{degree4}
Consider $\sigma_r(Ch_4(V^*))\subset S^4(V^*)$, where dim $V\geq4r$, prolongations enable one to find  modules in the ideal of $\sigma_r(Ch_4(V^*))$.
\begin{theorem}{\rm (restatement of Theorem \ref{secantrchow4})}
When  $\dim\ V\geq4r$, $$I_{4r+1}(\sigma_r(Ch_4(V^*)))=I_5(Ch_4(V^*))^{(4r-4)}$$ and
$$S_{(6,6,4^{4r-2})}V\subset I_{4r+1}(\sigma_r(Ch_4(V^*))).$$
\end{theorem}
\begin{proof}
By Proposition \ref{idealchow}, Proposition \ref{inj} and Proposition \ref{sur}, $I_4(Ch_4(V^*))=0$ and
$I_{5}(Ch_4(V^*))= S^{5} (S^4V)-S^{4} (S^{5}V)$, so $I_5(Ch_4(V^*))^c=S^4(S^5V)$.
By  Theorem \ref {sidman}, $$I_{4r+1}(\sigma_r(Ch_4V^*))=I_5(Ch_4V^*)^{(4r-4)}.$$
Consider the polarization map
$$F_{4r-4,5}:S^{4r+1}(S^4V)\rightarrow S^{4r-4}(S^4V)\otimes S^{5}(S^4V),$$
by Proposition \ref{propol},
$$I_5(Ch_4V^*)^{(4r-4)}=F_{4r+1,4}^{-1}(S^{4r-4}(S^4V)\otimes I_5(Ch_4(V^*))).$$
Since $S_{(6,6,6,2)}\subset S^4(S^5V)$ has the lowest highest weight vector with respect to the lexicographic order among all the modules in $S^4(S^5V)$, by the Littlewood-Richardson rule,
$$S_{(6,6,4^{4r-2})}V\subsetneq S^{4r-4}(S^4V)\otimes I_5(Ch_4(V^*))^c=S^{4r-4}(S^4V)\otimes S^4(S^5V).$$
Therefore by  Schur's lemma $$S_{(6,6,4^{4r-2})}V\subset I_5(Ch_4V^*)^{(4r-4)}=I_{4r+1}(\sigma_r(Ch_4(V^*)).$$

\end{proof}
\begin{remark}
Consider $r=2$ and $\dim\ V=8$. Since $\sigma_2(Ch_4\mathbb{C}^{4*}))$ is a proper subset $\mathbb{P}S^4(\mathbb{C}^{4*})$, by inheritance (see \cite{MR2865915}), the ideal of $\sigma_2(Ch_4(V^*))$  contains modules with length 4. So $S_{(6,6,4,4,4,4,4,4)}V$ is not enough to cut out $\sigma_2(Ch_4(V^*))$ set-theoretically. One can get a length 4 module with high degree in the ideal of  $\sigma_2(Ch_4(V^*))$ by Koszul Young flattenings \cite{2015arXiv151000886G}, but I still do not know whether they are enough to define $\sigma_2(Ch_4(V^*))$ set-theoretically. We know that $\dim\ S_{(6,6,4,4,4,4,4,4)}V=336$  and ${\rm codim}\ \sigma_2(Ch_3(V^*))=272$, therefore $\sigma_2(Ch_4(V^*))$ is far from
being a complete intersection. Obviously $\mathbb{P}S^4(\mathbb{C}^{7*})$ with dimension 210 is in the zero set of $S_{(6,6,4,4,4,4,4,4)}V$, while the dimension of $\sigma_2(Ch_4\mathbb{C}^{4*}))$ is 57,
The next question is:  what is the difference between the dimension of $\sigma_2(Ch_4(V^*))$ and the zero set of $S_{(6,6,4,4,4,4,4,4)}V$?
\end{remark}
\section{General case for  even degrees}\label{degreeeven}
Let $\lambda$ be a partition of order $kd$, recall a semi-standard tableau of  shape $\lambda$
and content $k\times d$ is a semi-standard tableau associated to $\lambda$ and filled with $\{1,\cdots,k\}$
such that each $i\in\{1,\cdots,k\}$  appears $d$ times.

\begin{proposition}{\rm \cite{2015arXiv151102927B}}\label{keypro} Let $\lambda$ be a partition with order $kd$ with $d$ odd , then
the multiplicity of $\lambda$ in $S^k(S^dV)$  is less than or equal to the number of semi-standard tableaux of shape $\lambda$
and content $k\times d$ with the additional property : for each pair $(i,j), 1\leq i\neq j\leq k$, the set of columns
of $i$ is not exactly the columns of $j$.
\end{proposition}
\begin{proposition}\label{plethysmmult}{\rm \cite{MR3349658}} Let $\lambda$ be a partition with order $kd$ and  let $u$ be even, then
$${\rm mult}(S_\lambda V,S^k(S^dV))={\rm mult}(S_{\lambda+(u^k)}V, S^k(S^{d+u}V)).$$
\end{proposition}
\begin{theorem}\label{Lowweight}
 $S_{((2m+2)^{2m-1},2)}V\subset S^{2m}(S^{2m+1}V)$, with multiplicity 1,
and $S_{((2m+2)^{2m-1},2)}V$ is the smallest module with respect to the lexicographic order among all the modules in the decomposition of $S^{2m}(S^{2m+1}V)$.
\end{theorem}
\begin{proof}First, let $\lambda=(\lambda_1,\cdots,\lambda_{2m})$ be a partition with order $4m^2+2m$ and smaller than $((2m+2)^{2m-1},2)$ with respect to the lexicographic order, then
$\lambda_1\leq2m+2$ and $\lambda_{2m}\geq3$.  Consider the semi-standard tableaux with content $2m\times(2m+1)$;  the first 3 columns must be filled
with $\{1,\cdots,2m\}$. Therefore there are $\binom{\lambda_1-3}{2m-2}\leq 2m-1$ possible sets of columns, but there are $2m$ numbers to be filled in the semi-standard tableaux, so by Proposition \ref{keypro}, ${\rm mult}(S_\lambda V,S^{2m}(S^{2m+1}V))=0.$

Second, consider the partition $\lambda=((2m+2)^{2m-1},2)$, by Proposition \ref{plethysmmult},\\ ${\rm mult}(S_\lambda V,S^{2m}(S^{2m+1}V))={\rm mult}(S_{(2m^{2m-1})}V,S^{2m}(S^{2m-1}V))$. By \cite{MR983608} formula (80),\\ ${\rm mult}(S_{(2m^{2m-1})}V,S^{2m}(S^{2m-1}V))=1$.
 The only filling is the following (I take m=3 as an example).
\begin{center}
\begin{tabular}{|c|c|c|c|c|c|}
    \hline
 1&1&1&1&1&2\\
\hline
	2&2&2&2&3&3\\
    \hline
    3&3&3&4&4&4\\
    \hline
    4&4&5&5&5&5\\
    \hline
    5&6&6&6&6&6\\
    \hline
    \end{tabular}
\end{center}
\end{proof}
Let $d=2m\geq4$ and $\dim\ V\geq2mr$, consider the variety  $\sigma_r(Ch_{2m}(V^*))\subset S^{2m}V^*$.
%A partition is an even partition if all the components of the partition are even numbers. When $d$ is even, any even partition with length no more %than $k$ has positive plethysm coefficients in $S^k(S^dV)$ \cite {MR2745569}.
\begin{theorem}{\rm (restatement of Theorem \ref{seccanteven})}
 The isotypic component of \\$ S_{({(2m+2)}^m,{(2m)}^{2mr-m})}V$ is contained in $I_{2m+1}(Ch_{2m}(V^*))^{(2m(r-1))}\subset I_{2mr+1}(\sigma_r(Ch_{2m}(V^*))).$
Moreover any module with even partition and  smaller than $((2m+2)^{2m-1},2)$ (with respect to the lexicographic order) is in $ I_{2mr+1}(\sigma_r(Ch_{2m}(V^*)))$.
\end{theorem}
\begin{proof}
By Theorem \ref{Lowweight}, $S_{((2m+2)^{2m-1},2)}V$ is the smallest module (with respect to the lexicographic order) in the decomposition of $S^{2m}(S^{2m+1}V)$.
Therefore by Corollary \ref{difference}, any module smaller than $S_{((2m+2)^{2m-1},2)}V$ (with respect to the lexicographic order) is not in $I_{2m+1}(Ch_{2m}(V^*))^c \subset S^{2m+1}(S^{2m}V)$.\\\\
Consider the polarization map
$$F_{2mr-2m,2m+1}:S^{2mr+1}(S^{2m}V)\rightarrow S^{2mr-2m}(S^{2m}V)\otimes S^{2m+1}(S^{2m}V).$$
By Proposition \ref{propol},
$$I_{2m+1}(Ch_{2m}(V^*))^{(2m(r-1))}=F_{2mr-2m,2m+1}^{-1}(S^{2mr-2m}(S^{2m}V)\otimes I_{2m+1}(Ch_{2m}(V^*))).$$
By the Littlewood-Richardson rule,
$$S_{((2m+2)^m,(2m)^{2mr-m})}V\subsetneq S^{2mr-2m}(S^{2m}V)\otimes I_{2m+1}(Ch_{2m}(V^*))^c.$$
 Moreover any module in $S^{2mr+1}(S^{2m}V)$ with even partition and smaller than $((2m+2)^{2m-1},2)$ is not contained in  $S^{2mr-2m}(S^{2m}V)\otimes I_{2m+1}(Ch_{2m}(V^*)))^c$.

Therefore by Schur's lemma the  isotypic component of
$S_{((2m+2)^m,(2m)^{2mr-m})}V$ is contained in $F_{2mr-2m,2m+1}^{-1}[S^{2mr-2m}(S^{2m}V)\otimes I_{2m+1}(Ch_{2m}(V^*))]=I_{2m+1}(Ch_{2m}(V^*))^{(2m(r-1))}.$\\\\
Moreover any module in $S^{2mr+1}(S^{2m}V)$ with even partition and smaller than $((2m+2)^{2m-1},2)$ (with respect to the lexicographic order) is in $ I_{2m+1}(Ch_{2m}(V^*))^{(2m(r-1))}$.

By Corollary \ref{proisec}, $I_{2m+1}(Ch_{2m}(V^*))^{(2m(r-1))}\subset I_{2mr+1}(\sigma_r(Ch_{2m}(V^*)))$, the results follow.
\end{proof}
%\begin{remark}
%We can compute the dimension of all the modules we obtain in $I_{2mr+1}(\sigma_r(Ch_{2m}(V^*)))$ and compare it to
%the codimension of $\sigma_r(Ch_{2m}(V^*))$.
%\end{remark}
\section{A property about Plethysm}\label{plethysm}
\begin{lemma}\label{duality}{\rm \cite{MR1651092,MR1190119,MR3349658}}
${\rm mult}(S_\lambda V, S^k(S^{2l}V))={\rm mult}(S_{\lambda^T} V, S^k(\Lambda^{2l}V))$, and ${\rm mult}(S_\lambda V,$\\$ S^k(S^{2l+1}V))={\rm mult}(S_{\lambda^T} V, \Lambda^k(\Lambda^{2l}V))$.
\end{lemma}
\begin{theorem}\label{plethysmeven}
Let d be even,
if $S_{(a_1,\cdots,a_p)}\subset S^k(S^dV)$ and $S_{(b_1,\cdots,b_q)}\subset S^l(S^dV)$
with $a_p\geq b_1$, then
$$S_{(a_1,\cdots,a_p,b_1,\cdots,b_q)}\subset S^{k+l}(S^dV)$$
as long as $\dim\ V\geq k+l$.
\end{theorem}
\begin{proof}
Let $\lambda=(a_1,\cdots,a_p)$ and $\mu=(b_1,\cdots,b_q)$. By Lemma \ref{duality},
${\rm mult}(S_{\lambda^T} V, S^k(\Lambda^{d}V))>0$ and ${\rm mult}(S_{\mu^T} V, S^l(\Lambda^{d}V))>0$,
so ${\rm mult}(S_{\lambda^T+\mu^T} V, S^{k+l}(\Lambda^{d}V))>0$. By Lemma \ref{duality} again, $${\rm mult}(S_{(\lambda, \mu)}V, S^{k+l}(S^{d}V))>0.$$
\end{proof}
\begin{remark}
This is false when $d$ is odd: C.Ikenmeyer gave a counter-example for $d=3$.
There exists $k_0$ such that $S_{6^{k_0}}V\subset S^{2k_0}({S^3V})$ but
$S_{6^{k_0+1}}V\subsetneq S^{2k_0+2}({S^3V})$.
\end{remark}
\section{Appendix}\label{appendix}
\subsection{ ${\mathbf{P}}$ versus ${\mathbf{NP}}$ problem}
Informally speaking, the ${\mathbf{P}}$ versus ${\mathbf{NP}}$ problem (see e.g.\cite{sipser}) asks whether every problem whose solution can be quickly verified by a computer can also be quickly solved by a computer. An early mention of it was a 1956 letter written by Kurt G\"{o}del to John von Neumann. G\"{o}del asked whether a certain problem could be solved in quadratic or linear time \cite{Hartmanis:1989:GNPa}. The precise statement of the P versus NP problem was introduced in 1971 by Stephen Cook in \cite{Cook:1971:CTP:800157.805047} and is considered to be the most important open problem in theoretical computer science \cite{Fortnow:2009:SPV:1562164.1562186}.

In computational complexity theory, a {\it decision problem} is a question in some formal system with a yes-or-no answer, depending on the values of input parameters. The class ${\mathbf{P}}$ consists of all those decision problems that can be solved in an amount of time that is polynomial in the size of the input; the class ${\mathbf{NP}}$ consists of all those decision problems whose positive solutions can be verified in polynomial time given the right information. For example, given a set $A$ of $n$ integers and a subset $B$ of $A$, the statement that \lq\lq$B$ adds up to zero\rq\rq can be quickly verified with at most $(n-1)$ additions. However, there is no known algorithm to
find a subset of A adding up to zero in polynomial time.
\subsection{Valiant's conjecture}
\begin{definition}
An ${\it arithmetic\ circuit}$ $\mathcal{C}$ over $\mathbb{C}$ and the set of variables $\{x_1,...,x_N\}$ is a directed acyclic graph  with vertices
of in-degree 0  and exactly one vertex of out-degree 0. Every vertex in it with in degree zero is called an input gate and is labeled by either a variable $x_i$ or an element in $\mathbb{C}$. Every other gate is labeled by either $+$ or $\times$, exactly one vertex of out-degree 0.
\end{definition}
 A circuit has two complexity measures associated with it: size and depth. The ${\it size}$ of a circuit is the number of gates in it, and the ${\it depth}$ of a circuit is the length of the longest directed path in it.
\begin{proposition}On an arithmetic circuit $\mathbb{C}$, each gate computes
a  polynomial.    The polynomial computed
by the output gate is denoted by $P_C$ and called the
polynomial defined by the circuit.

\end{proposition}
\begin{definition}
The class ${\mathbf{VP}}$ consists of sequences of polynomials $(p_n)$ of polynomial of degree $d(n)$ and variables $v(n)$, where $d(n)$ and $v(n)$ are bounded by polynomials in $n$ and such that
there exists a sequence of arithmetic circuits $\mathcal{C}_n$ of polynomially bounded size such that
$\mathcal{C}_n$ defines $p_n$.
\end{definition}
\begin{example}
The sequence $({\det}_n)\in {\mathbf{VP}}$, where $det_n$ denotes the determinant of a $n\times n$ matrix.
 \end{example}
\begin{definition}
Consider a sequence $h=(h_n)$ of polynomials in variables $x_1,\cdots,x_n$ of the form $$h_n=\sum_{e\in\{0,1\}^n}g_n(e)x_1^{e_1}\cdots x_n^{e_n},$$
where $(g_n)\in {\mathbf{VP}}$.
The class $\mathbf{VNP}$ is defined to be the set of all sequences the form $h$.
\end{definition}
\begin{definition}
A problem P is hard for a complexity class $\mathbf{C}$ if all problems in  $\mathbf{C}$ can
be reduced to P (i.e. there is an algorithm to translate any instance of a problem in  $\mathbf{C}$
to an instance of P with comparable input size). A problem P is complete for  $\mathbf{C}$ if it is
hard for  $\mathbf{C}$ and P $\in\mathbf{C}$.
\end{definition}
\begin{proposition}{\rm \cite{MR526203}}
 The sequence $({\perm}_n)$ is $\mathbf{VNP}$-complete.
 \end{proposition}

Therefore to prove Valiant's Conjecture $\mathbf{VP\neq VNP}$ \cite{vali:79-3}, we only need to prove there does not exist a polynomial size
circuit computing the permanent.
\bibliographystyle{amsplain}
%\bibliography{thesis}
%\begin{thebibliography}{59}
%\bibitem{}J.M.Landsberg, \textsl{Tensors: Geometry and Applications}, Graduate Studies in Mathematics,vol.128, AMS, Providence, 2011.
%\bibitem{}J.M.Landsberg, \textsl{Geometric Complexity Theory: An introduction for Geometers}, arXiv1305.7387L, 2013.
%\bibitem{}Ankit Gupta, Pritish Kamath, Neeraj Kayal, Ramprasad Saptharishi, \textsl{Approaching the chasm at depth four}, Electronic Colloquim
%on Computational Complexity, Revision 3 of Report No.98(2012).
%\bibitem{}P. Gordon, \textsl{Das Zerfallen der Curven in gerade Linien}, Math Ann(1894),no. 45, 410-427
%\bibitem{}1. J. Hadamard, \textsl{M\'emoire sur l\'elimination}, Acta Math. 220(1897), no. 1, 201¨C238. MR 1554881
%\end{thebibliography}
\bibliography{Lmatrix}
%{MR0453768, MR1054015, MR1777172}
\end{document}